\documentclass[10pt]{amsart}

\usepackage{amsfonts,amsmath,amssymb}

\usepackage{enumerate}

\numberwithin{equation}{section}

\newtheorem{theorem}{Theorem}[section]

\newtheorem{lemma}[theorem]{Lemma}

\newtheorem{definition}{Definition}[section]

\newtheorem{corollary}[theorem]{Corollary}

\newtheorem{remark}[theorem]{Remark}

\newcommand{\cl}[1]{\mathcal{#1}} %\cl A

\newcommand{\bb}[1]{\mathbb{#1}}

\begin{document}

\title{On synthetic and transference properties of group homomorphisms }

\author[G.~K.~Eleftherakis]{G. K. Eleftherakis}

\address{G. K. Eleftherakis\\ University of Patras\\Faculty of Sciences\\
Department of Mathematics\\265 00 Patra Greece }

\email{gelefth@math.upatras.gr}



\date{}

\maketitle

\begin{abstract} 
We study Borel homomorphisms $\theta : G\rightarrow H$ for 
arbitrary locally compact second countable groups $G$ and $H$ for which 
the measure $$\theta _*(\mu )(\alpha )=\mu (\theta ^{-1}(\alpha ))\quad \text{for } \quad \alpha \subseteq H \text{\;\;a \;\;Borel set } $$
 is absolutely continuous with respect to $\nu,$ where $\mu $ (resp. $\nu $)  is a Haar measure for $G,$ (resp. $H$). 
 We define a natural mapping $\cl G$ from the class of maximal abelian selfadjoint 
algebra bimodules (masa bimodules) in $B(L^2(H))$ 
into the class of masa bimodules in $B(L^2(G))$ and we use it to prove that if $k\subseteq G\times G$  is a set of operator synthesis, then 
$(\theta \times \theta)^{-1} (k)$ is also a set of operator synthesis and if $E\subseteq H$ is a set of local synthesis for the 
Fourier algebra $A(H)$, then  
$\theta ^{-1}(E)\subseteq G$ is a set of local synthesis for $A(G).$ We also prove that if  $\theta ^{-1}(E)$ is an $M$-set (resp. $M_1$-set), then $E$ 
is an $M$-set (resp. $M_1$-set) and if $Bim(I^\bot )$ is the masa bimodule generated by the annihilator of the ideal $I$ in $VN(G)$, then there exists an ideal $J$ 
such that $\cl G(Bim(I^\bot ))=Bim(J^\bot ).$ If this ideal $J$ is an ideal of multiplicity then $I$ is an ideal of multiplicity.
In case $\theta _*(\mu )$ is a Haar measure for $\theta (G)$ 
we show that $J$ is equal to the ideal $\rho _*(I)$ generated by $\rho (I),$ 
where $\rho (u)=u\circ \theta , \;\;\forall \;u\;\in \;I.$
\end{abstract}

\section{Introduction}
Arveson discovered the connection between spectral synthesis and operator synthesis, \cite{arv}. Froelich found 
the precise connection for separable abelian groups, \cite{F}, and Spronk and Turowska for separable compact groups, \cite{SSTT}. 
Ludwig and Turowska generalized the previous 
results in 
the case of locally compact second countable groups, $G.$ They proved that if $E\subseteq G$ is a closed set and $E^*=\{(s,t)\in G\times G: ts^{-1}\in E \}$,
then $E$ is a set of local synthesis if and only if $E^*$ is a set of operator synthesis, \cite{lt}. Anoussis, Katavolos and Todorov stated in \cite{akt} 
that given a closed ideal $I$ of the Fourier algebra $A(G),$ where $G$ is a locally compact second countable group, there are two natural ways to construct a $w^*$-closed 
maximal abelian selfadjoint (masa) bimodule:

(i) Let $I^\bot $ be the annihilator of $I$ in $VN(G)$ and then take the masa bimodule $Bim(I^\bot )$ in the space of bounded operators acting on $L^2(G), 
B(L^2(G)),$ generated by $I^\bot .$

(ii) Consider the space $Sat(I)=\overline{span\{N(I)T(G)\}}^{\|\cdot\|_{T(G)}},$ where $N(u)(s,t)=u(ts^{-1})$ for all $u\in I$ 
and $T(G)$ is the projective tensor product $L^2(G)\hat \otimes L^2(G)$ and then take its annihilator $Sat(I)^\bot $ in $B(L^2(G)).$ 

One of their main results is that  $Bim(I^\bot )=Sat(I)^\bot .$ They used this in order to prove that if $A(G)$ possesses an approximate identity, then 
$E\subseteq G$ is a set of spectral synthesis if and only if $E^*$ is a set of operator synthesis.

The transference of results from Harmonic Analysis to Operator Theory and vice versa is not limited to the case of synthesis. In \cite{stwt},
Shulman, Todorov and Turowska proved that if $G$ is a locally compact second countable group and $E\subseteq G$, then $E$ 
is  an $M$-set (resp. $M_1$-set) if and only if $E^*$ is an $M$-set (resp. $M_1$-set). Subsequently, Todorov and Turowska in \cite{tt} proved that an ideal $J\subseteq A(G)
$ is an ideal of multiplicity if and only if $Bim(J^\bot )$   is an operator space of multiplicity.

In Section \ref{2222}, we consider  arbitrary locally compact second countable groups $G$ and $H,$ Borel homomorpisms $\theta : G\rightarrow H$ for which 
the measure $$ \theta _*(\mu )(\alpha )=\mu (\theta ^{-1}(\alpha )) \quad \text{for}\;\alpha \subseteq H\;\text{ a Borel set} $$ 
is absolutely continuous with respect to $\nu ,$ $(\theta _*(\mu )<<\nu ),$ where $\mu $ (resp. $\nu $) is a Haar measure for $G$ (resp. $H$). 
Recall that Borel measurable homomorphisms between locally compact groups are automatically continuous,
 \cite{kl},   \cite{mm}. We define a natural mapping $\cl G$ from 
the class of masa  bimodules in $B(L^2(H))$ to the class of masa  bimodules in $B(L^2(G)).$ We prove that 
$$ \cl G( M_{max}(k) )= M_{max}((\theta \times \theta)^{-1}(k) ), \;\;\;\cl G(M_{min}(k))=M_{min}((\theta \times \theta)^{-1}(k) ), $$
where   $M_{max}(k)$ is the biggest masa bimodule supported on the $\omega$-closed $k,$ see the definition below, and  $M_{min}(k)$  
is the smallest. Therefore if $M_{max}(k)$  is a synthetic operator space, then $M_{max}((\theta \times \theta)^{-1}(k) )$ is operator synthetic. This implication can also be deduced
from theorem 4.7 of \cite{st} or from theorem 5.2 of \cite{ele2}. Here we present a different proof. We also prove that if $E\subseteq H$ 
is a set of local synthesis, then $\theta ^{-1}(E)$ is   a set of local synthesis and if $\cl U$ 
is a $w^*$-closed masa bimodule for which $\cl G(\cl U)$ contains a non-zero compact operator (or a non-zero finite rank operator or a rank one operator),
so does $\cl U.$ We use this result to prove that if $\theta ^{-1}(E)$ is an $M$-set (resp. $M_1$-set), then $E$ 
is an $M$-set (resp. $M_1$-set). If $I$ is an ideal of $A(H)$, we prove that there exists an ideal $J\subseteq A(G)$ 
such that 
\begin{equation}\label{ideal} \cl G(Bim(I^\bot ))=Bim(J^\bot ),\;\;\;\;Sat(J)=\overline{span \{ N(\rho (I))T(G) \} }^{\|\cdot\|_{T(G)}}, 
\end{equation}
$$\rho (u)=u\circ \theta ,\;\;\forall \;u\;\in \;I, \;\;\;N(\rho (u))(s,t)=\rho (u)(ts^{-1}).$$ We use equalities (\ref{ideal}) to prove 
that if $J$ is an ideal of multiplicity, then $I$ is an ideal of multiplicity.   In Section \ref{3333} we assume that the measure 
$\theta _*(\mu )$ is  a Haar measure for the group $\theta (G).$ We prove that if $I$ is a closed ideal of $A(H)$, then $\rho (I)\subseteq A(G)$ and so we can choose in (\ref{ideal}) 
as $J$ the ideal $\rho _*(I)$ generated by  $\rho(I).$ We also prove that if $A(G)$ possesses an approximate identity and $E$ 
is an ultra strong Ditkin  set, then $\theta ^{-1}(E)$ is also an ultra strong  Ditkin set.

We now present the definitions and notation that will be used in this paper. If $S$ is a subset of a linear space, we denote by $[S]$  
its linear span. If $H$ and $K$ are Hilbert spaces, $B(H, K)$ is the set of bounded 
operators from $H$ to $K.$ We write $B(H)$ for $B(H, H).$ If $\cl X\subseteq B(H, K)$ is a subspace,
we write $Ref(\cl X)$ for the reflexive hull of $\cl X,$ that is,
$$Ref(\cl X)=\{T\in B(H, K): T\xi \in \overline{\cl X\xi }, \;\;\;\forall \;\xi \;\in \;H\}.$$

Let $G$ be a locally compact group with Haar measure $\mu ,$ and $T(G)$ the projective tensor product 
$L^2(G)\hat \otimes L^2(G).$ Every element $h\in T(G)$ is an absolutely convergent series, $$ h=\sum_i f_i\otimes g_i , \;\;
f_i,g_i\;\in \; L^2(G),\;\;i\;\in \;\bb N,$$
where $ \sum_i\|f_i\|_2\|g_i\|_2<+\infty .$ Such an element may be considered as a function $h: G\times G\rightarrow 
\bb C,$ defined by  $h(s,t)=\sum_i f_i(s)g_i(t).$ The norm in $T(G)$  is given by 
 $$ \|h\|_t=\inf\{\sum_i\|f_i\|_2\|g_i\|_2: h=\sum_i f_i\otimes g_i \}.$$
The space $T(G)$ is predual to $B(L^2(G)).$ The duality is given by 
$$(T,h)_t=\sum_i(T(f_i), \overline{g_i}),$$  where $h$ is as above and $(\cdot,\cdot)$ is the inner product of $L^2(G).$ 

A subset $F\subseteq G\times G$ is called marginally null if $F\subseteq (\alpha \times G)\cup (G\times \beta ),$  
where $\alpha$\ and $\beta $ are Borel sets such that $\mu (\alpha )=\mu (\beta)=0. $ In this case we 
write $F\simeq \emptyset.$ If $F_1$ and $F_2$ are subsets of $G\times G$, we write $F_1\simeq F_2$ 
if the symmetric difference $F_1\bigtriangleup F_2$ is marginally null. If $F\subseteq G\times G$, 
we denote by $M_{max}(F)$ the subspace of $B(L^2(G))$ consisting of all those operators $T$ satisfying 
$$(\alpha \times \beta )\cap F\simeq \emptyset\Rightarrow P(\beta )TP(\alpha )=0.$$ Here $P(\beta ),$  
and similarly $P(\alpha  )$ 
is the projection onto $L^2(\beta ,\mu ).$ 
We usually identify the algebra $ L^\infty (G, \mu ) $ with the algebra of operators 
$$M_f: L^2(G) \rightarrow  L^2(G) , g\rightarrow fg,$$ where $f\in L^\infty (G, \mu ) .$ 
This algebra is a maximal abelian selfadjoint algebra, referred to as ``masa'' in what follows. If $F\subseteq G\times G$, then $M_{max}(F)$ 
is an $L^\infty (G)$-bimodule. An $L^\infty (G)$-bimodule will be referred to as a ``masa bimodule.''
The space $ M_{max}(F) $ is reflexive. Also, there exists a $w^*$-closed masa bimodule $\cl U_0$ 
with the property that it is the smallest $w^*$-closed masa bimodule $\cl U$ such that 
$Ref(\cl U)= M_{max}(F) .$ We write $\cl U_0=M_{min}(F).$ 
Given a reflexive masa bimodule $\cl V$, there exists a set $k\subseteq G\times G$ which is marginally equal to 
$(\cup _{n\in \bb N}\alpha _n\times \beta _n)^c,$ where $\alpha _n, \beta _n$  are Borel subsets of $G$ 
such that $\cl V=M_{max}(k). $ An operator $T$ belongs to $\cl V$ if and only if 
$P(\beta _n)TP(\alpha _n)=0, \;\;\forall \;n.$
A set $k$ that is marginally equal to a complement of a countable union of 
Borel rectangles is called an $ \omega $-closed set.

An $\omega $-closed set $k$ is called operator synthetic if 
$M_{max}(k)=M_{min}(k).$

If $s\in G,$ we denote by $\lambda _s$ the operator given by 
$$\lambda _s(f)(t)=f(s^{-1}t),\;\;\;\forall \;f\;\in \;L^2(G).$$
The homomorphism $G\rightarrow B(L^2(G)): s\rightarrow \lambda _s$ is 
called the left regular representation.  We denote by $A(G)$ the set of maps  
$u:G\rightarrow \bb C$ given by $ u(s)=(\lambda _s(\xi ), \eta ) $ for $\xi ,\eta \in L^2(G).$ For any $u\in A(G)$, we write  
$$\|u\|_{A(G)}=\inf\{\|\xi \|_2\|\eta \|_2:u(s)=(\lambda _s(\xi ), \eta ) \;\forall \;s \}.$$
The set $A(G)$ is an algebra under the usual multiplication, and $\|\cdot\|_{A(G)}$ is  a norm making $A(G)$ a commutative 
regular semisimple Banach algebra. We call this algebra  a Fourier algebra. We denote by $VN(G)$ the following von Neumann subalgebra 
of $B(L^2(G))$:
$$VN(G)=[\lambda _s: s\in G]^{-w^*} .$$ 
 This algebra is the dual of the Fourier algebra $A(G).$ 
The duality is given by $(\lambda _s,u)_\alpha =u(s)$ for all $u\in A(G)$ and $s\in G.$

If $E\subseteq G$ is a closed set, we write 
$$I(E)=\{u\in A(G): u|_E=0\},$$
$$J_0(E)=\{u\in A(G): \exists \;\Omega\;\; open\;\;set,\;E\subseteq \Omega , \;\;u|_\Omega =0 \}.,$$
and $J(E)$ for the closure of $J_0(E)$ in $A(G).$ The spaces $I(E) $ and $ J(E)$ are closed ideals of $A(G)$ 
and $J(E)\subseteq I(E).$ The set $E$ is called a set of spectral synthsesis  if $J(E)= I(E).$ Let 
$I^c(E)$ be the set of all compactly supported functions $f\in I(E).$ We say that $E$ 
is a set of local spectral synthesis if $I^c(E)\subseteq J(E).$ 

If $u: G \rightarrow L$ is an arbitrary map, we write $N(u)$ for the map 
$$N(u): G\times G\rightarrow L, \;\;(s,t)\rightarrow u(ts^{-1}).$$ If $u\in A(G)$,
the map $N(u)$ can be written as $$N(u)=\sum_{i\in \bb N}\phi_i \otimes \psi_i,$$
where $\phi _i, \psi_i: G\rightarrow \bb C $ are Borel maps  satisfying $$ \|\sum_{i\in \bb N}|\phi _i|^2\|_\infty <+\infty  , \;\;\;
\|\sum_{i\in \bb N}|\psi _i|^2\|_\infty <+\infty  .$$ The map $N(u)$ satisfies $N(u)T(G)\subseteq T(G).$ See in \cite{spronk} 
for more details. 

If $I$ is a closed ideal of $A(G), I^\bot $ is its annihilator in $VN(G):$
$$I^\bot =\{T\in VN(G): (T,u)_\alpha =0,\;\;\forall \;u\;\in \;I\}.$$
We also write $$Sat(I)=[N(I)T(G)]^{-\|\cdot\|_t}\subseteq T(G).$$
If $\cl X$ is a subspace of $VN(G)$, we write $Bim(\cl X)$ for the folowing subspace of $B(L^2(G)):$
$$Bim(\cl X)=[M_\phi XM_\psi : X\in \cl X, \;\phi , \psi \in L^\infty (G)] ^{-w^*} .$$

\section{Synthetic and transference properties of group homomorphisms }\label{2222}

In this section we assume that  $G$ and $ H$ are locally compact second countable groups with Haar measures $\mu $ and $\nu $
 respectively, $\theta :G\rightarrow H$ is  a continuous homomorphism and $\theta _*(\mu )<<\nu. $ 
We conclude that the map 
$$\hat \theta : L^\infty (H)\rightarrow L^\infty (G), \;\;\;\hat \theta (f)=f\circ \theta $$ 
is a weak*-continuous homomorphism. If $\alpha \subseteq H,$ (resp. $\beta \subseteq G$) is a Borel set,
we denote by $P(\alpha ),$ (resp. $Q(\beta ))$ the projection onto $L^2(\alpha , \nu )$ (resp. $L^2(\beta ,\mu )$).
 If $\phi \in L^\infty (H)$ (resp. $\psi \in L^\infty (G)), $ we denote by 
$M_\phi$ (resp. $M_\psi $) the operator $L^2(H)\rightarrow L^2(H): f\rightarrow f\phi$
(resp. $L^2(G)\rightarrow L^2(G): g\rightarrow g\psi $ ). We define the following ternary ring of operators (TRO):
$$\cl N=\{X: XP(\alpha )=Q(\theta ^{-1}(\alpha ))X, \text{for}\;\alpha \subseteq H,\;\text{a Borel set}\}.$$
(For the definition and properties of TROs, see \cite{bm}). 
Observe that for every $\phi \in L^\infty (H)$ and $ X\in \cl N$, we have $XM_\phi = M_{\phi \circ \theta } X.$ Suppose that 
$Ker(\hat \theta )=L^\infty (\alpha _0^c),$ for some Borel set $\alpha _0^c\subseteq H.$ Then the map 
$$L^\infty (\alpha _0)\rightarrow L^\infty (G), \;\;\;\hat \theta (f|_{\alpha _0})=f\circ \theta , \;\;f\in L^\infty (H)$$ 
is  a one-to-one $*$-homomorphism. We now define the following TRO:
$$\cl M=\{X: XP(\alpha )=Q(\theta ^{-1}(\alpha ))X, \;\;\alpha \subseteq \alpha _0,\;\;Borel\}\subseteq B(L^2(\alpha _0), L^2(G)).$$
If $R$ is the projection onto $L^2(\alpha _0),$ we can easily see that $\cl N=\cl M R.$

Let $\cl A\subseteq B(L^2(G))$ be the commutant of the algebra $\{M_{\phi \circ \theta } : \;\phi \in \; L^\infty (H)\}.$
By theorem 3.2 of \cite{ele}, 
$$[\cl M^*\cl M] ^{-w^*} =L^\infty (\alpha _0),\;\;\; [\cl M\cl M^*] ^{-w^*}  =\cl A.$$
For every masa bimodule $\cl U\subseteq B(L^2(H)),$ we define 
$$\cl G(\cl U)=[\cl N\cl U\cl N^*] ^{-w^*} $$ and for every masa bimodule $\cl U\subseteq B( L^2(\alpha _0) ),$ we define 
$$ \cl F(\cl U) =[\cl M\cl U\cl M^*] ^{-w^*}. $$ 
By proposition 2.11 of \cite{ele}, the map $\cl F$ is a biijection from the masa bimodules acting on $L^2(\alpha _0) $ 
onto the $\cl A$-bimodules acting on $L^2(G). $ The inverse of $\cl F$ is given by 
 $$\cl F^{-1}(\cl V)=[\cl M^*\cl V\cl M] ^{-w^*}. $$ We can easily see that 
$$\cl G(\cl U)=\cl F(R\cl UR).$$
\begin{remark} \em{We inform the reader of the following:

(i) The spaces $\cl U$ and $ \cl F(\cl U) $ are stably isomorphic in the sense that  there exists a Hilbert space $H$ 
such that the spaces $ \cl U\bar \otimes B(H ) $ and $ \cl F(\cl U)\bar \otimes B(H ) $ are isomorphic as dual 
operator spaces, where $\bar \otimes $ is the normal spatial tensor product, \cite{ept}.

(ii)  The spaces $\cl U, \cl F(\cl U) $ are spatially Morita equivalent in the sense of \cite{ele2}.}
\end{remark}

\begin{lemma}\label{21} Suppose $k\subseteq \alpha_0 \times \alpha _0$ is an $\omega $-closed set.  Suppose further that $\cl U= M_{max}(k)$
and $\cl V=M_{max}(\sigma )$, where $\sigma =(\theta \times \theta) ^{-1}(k)$. Then  $\cl F(\cl U)=\cl V.$
\end{lemma}
\begin{proof} Suppose that $\alpha_n\subseteq  \alpha _0$ and $\beta _n\subseteq \alpha_0, n\in \bb N$ are Borel sets such that 
$k=(\cup _n(\alpha_n \times \beta_n ))^c$. Then $\sigma =(\cup _n(\theta ^{-1}(\alpha_n )\times \theta ^{-1}(\beta_n) ))^c.$  
If $Z\in \cl U, X, Y\in \cl M,$ then 
$$ Q(\theta ^{-1}(\beta _n)) XZY^*Q(\theta ^{-1}(\alpha _n))=XP(\beta _n)ZP(\alpha _n) Y^*=X0Y^*=0, \;\;\forall n.$$
Therefore $\cl M\cl U\cl M^*\subseteq  \cl V.$ Similarly, we can prove $\cl M^*\cl V\cl M\subseteq  \cl U.$  
The above relations imply that $$\cl M\cl M^*\cl V\cl M\cl M^*\subseteq \cl M\cl U\cl M^*\subseteq \cl V.$$
Since $[\cl M\cl M^*] ^{-w^*}  $ is an unital algebra, $$\cl V=[ \cl M \cl U \cl M^* ]^{-w^*}=\cl F(\cl U).$$
\end{proof}

\begin{lemma}\label{22} Let $\cl U,$ be as in Lemma \ref{21}, Then  $\cl F(\cl U_{min})=\cl F(\cl U)_{min}.$
\end{lemma}
\begin{proof}
If $\cl W=M_{max}(\Omega )$ is a reflexive masa bimodule, we write  $\cl W_{min}=M_{min}(\Omega ).$ 
From Lemma \ref{21} $\cl F(\cl U)=M_{max}(\sigma ).$ Therefore   $\cl F(\cl U)_{min}=M_{min}(\sigma ).$ Since 
$\cl A$ contains the masa $L^\infty (G)$ and $\cl F(\cl U)$ is a reflexive $\cl A-$bimodule the space $$\Pi= \left(\begin{array}{clr} \cl A & \cl F(\cl U)
 \\ 0 & \cl A\end{array}\right)  $$ is a CSL algebra. 
From the proof of proposition 4.7 of \cite{ele}, we have 
$$\left(\begin{array}{clr} 0 & \cl F(\cl U)_{min} \\ 0 & 0\end{array}\right) =
\left(\begin{array}{clr} 0& \cl F(\cl U) \\ 0 & 0\end{array}\right)_{min}\subseteq \Pi _{min}.$$ 
Also the diagonal of $\Pi ,$ $\left(\begin{array}{clr} \cl A & 0 \\ 0 & \cl A\end{array}\right),$ belongs to $ \Pi _{min} .$ 
Thus, $$\left(\begin{array}{clr} \cl A & \cl F(\cl U)_{min} \\ 0 & \cl A\end{array}\right)\subseteq \Pi _{min} .$$ 
But   $$Ref\left(\left(\begin{array}{clr} \cl A & \cl F(\cl U)_{min} \\ 0 & \cl A\end{array}\right)\right)=\Pi .$$ 
Therefore, $$\Pi _{min}\subseteq  \left(\begin{array}{clr} \cl A & \cl F(\cl U)_{min} \\ 0 & \cl A\end{array}\right) .$$ 
We conclude that $$\Pi _{min}= \left(\begin{array}{clr} \cl A & \cl F(\cl U)_{min} \\ 0 & \cl A\end{array}\right) .$$
Since $\Pi $ is an algebra  by \cite[Theorem 22.19]{dav}, $\Pi _{min}$ is also an algebra, which implies $\cl A \cl F(\cl U)_{min} \cl A\subseteq \cl F(\cl U)_{min} .$

Observe that 
$$\cl M \cl U_{min} \cl M^* \subseteq \cl M \cl U \cl M^* \subseteq \cl F(\cl U) .$$
Since by Lemma \ref{21} $\cl F(\cl U)$ is a reflexive space, 
 $$ Ref(\cl M \cl U_{min} \cl M^*) \subseteq \cl F(\cl U) .$$
Let $Z\in \cl F(\cl U)$ and assume that $Z$ does not belong to $Ref(\cl M \cl U_{min} \cl M^*).$ 
Thus, there exists a $\xi \in L^2(G)$  such that $Z\xi $ does not belong to $\overline{[\cl M \cl U_{min} \cl M^*\xi ]}.$  
Thus,  there exists an $\omega  \in L^2(G)$  such that $(XSY^*\xi ,\omega )=0, \;\;\forall \; X, Y \;\in \cl M, S\in \cl U_{min}$ 
and $(Z\xi ,\omega )\neq 0.$ We have $(SY^*\xi ,X^*\omega )=0, \;\;\forall \; X, Y \;\in \cl M, S\in \cl U_{min}.$ 
Since $SY^*\xi \in \overline{\cl U_{min}Y^*\xi } =\overline{\cl UY^*\xi }$, we have
$$(SY^*\xi ,X^*\omega )=0, \;\;\forall \; X, Y \;\in \cl M, S\in \cl U\Rightarrow 
(XSY^*\xi ,\omega )=0, \;\;\forall \; X, Y \;\in \cl M, S\in \cl U.$$ 
Since $ \cl F(\cl U) =[\cl M\cl U\cl M^*] ^{-w^*}, $  
$$(T\xi ,\omega )=0,\;\;\forall \;T\;\in \cl F(\cl U) .$$
Therefore $(Z\xi ,\omega )= 0.$ This contradiction shows that 
$$  \cl F(\cl U) \subseteq Ref( \cl M \cl U_{min} \cl M^* )  \Rightarrow \cl F(\cl U)=
Ref(\cl M \cl U_{min} \cl M^*) .$$ Since $   [\cl M \cl U_{min} \cl M^*]^{-w^*}  $  
is a masa bimodule,
$  \cl F(\cl U)_{min}\subseteq [\cl M \cl U_{min} \cl M^*]   ^{-w^*}  .$  By symmetry, we have 
$\cl U_{min}\subseteq [\cl M ^*\cl F(\cl U)_{min} \cl M ] ^{-w^*}.$  Thus, 
$$ \cl F(\cl U)_{min} \subseteq [\cl M \cl U_{min} \cl M^*]   ^{-w^*} \subseteq  
[\cl M\cl M ^*\cl F(\cl U)_{min} \cl M\cl M^* ] ^{-w^*}\subseteq [\cl A\cl F(\cl U)_{min} \cl A ] ^{-w^*}\subseteq \cl F(\cl U)_{min} . $$
Therefore $\cl F(\cl U_{min})=\cl F(\cl U)_{min}.$

\end{proof}
\begin{theorem}\label{23} Let $k\subseteq H\times H$ be an $\omega $-closed set. Then 

(i)  $\cl G(M_{max}(k))=M_{max}((\theta \times \theta )^{-1}(k))$ and

(ii)  $\cl G(M_{min}(k))=M_{min}((\theta \times \theta )^{-1}(k)).$
\end{theorem}
\begin{proof}  

(i) By Lemma \ref{21}, 
$$ \cl F(M_{max}( k\cap (\alpha_0 \times \alpha _0) ) =M_{max}
((\theta \times \theta )^{-1}(k\cap (\alpha_0 \times \alpha _0) )).$$
We can easily see that if $k_1, k_2$ are $\omega -$closed sets then $$M_{max}(k_1\cap k_2)=M_{max}(k_1)\cap M_{max}(k_2),$$ thus 
$$M_{max}((\theta \times \theta )^{-1}(k\cap (\alpha_0 \times \alpha _0) ))=M_{max}((\theta \times \theta )^{-1})(k))\cap 
M_{max}( (\theta \times \theta )^{-1}(\alpha_0 \times \alpha _0) ). $$
Since $\theta ^{-1}(\alpha _0)=G$ up to measure zero, the sets $G\times G, (\theta \times \theta )^{-1}(\alpha_0 \times \alpha _0) $ 
are marginally equal, thus $ M_{max} ((\theta \times \theta )^{-1}(\alpha_0 \times \alpha _0) )=B(L^2(G)).$
Therefore $$\cl G ( M_{max}(k) )=\cl F(M_{max}( k\cap (\alpha_0 \times \alpha _0) ) =M_{max}((\theta \times \theta) ^{-1}(k)).$$

(ii) By Lemma \ref{22}, 
$$\cl F(M_{min}( k\cap (\alpha_0 \times \alpha _0) )=M_{min}(
(\theta \times \theta )^{-1}(k\cap (\alpha_0 \times \alpha _0) )=M_{min}((\theta \times \theta )^{-1}(k)\cap (\theta \times \theta )^{-1}
(\alpha_0 \times \alpha _0) ).$$ Since the sets $G\times G, (\theta \times \theta )^{-1}(\alpha_0 \times \alpha _0) $ 
are marginally equal, we conclude that  $$\cl G(M_{min}(k)) =\cl F(M_{min}( k\cap (\alpha_0 \times \alpha _0) )=M_{min}((\theta \times \theta )^{-1})(k)).$$

\end{proof}

\begin{corollary}\label{24}  If $\cl U=M_{max}(k)  $ is a synthetic masa bimodule acting on $L^2(H)$,
then $\cl G(\cl U)=M_{max}((\theta \times \theta)^{-1}(k) ) $ is also synthetic.
\end{corollary}

\begin{remark}\em{ The implication of the previous corollary was first proved in \cite[Theorem 4.7]{st}. 
In the present paper, we have given a different proof.}
\end{remark}

\begin{corollary}\label{28} Let $E\subseteq H$ be a closed set. Then

(i) $$ \cl G(M_{max}(E^*))= M_{max}(\theta ^{-1}(E)^*) \quad \text{and} \quad \cl G(M_{min}(E^*))=M_{min}(\theta ^{-1}(E)^*) ;$$

(ii) If $E$ is a set of local synthesis, then $\theta ^{-1}(E)$  is a set of local synthesis;

(iii) If $E$ is a set of local synthesis and $A(G)$ possess an approximate identity, then $\theta ^{-1}(E)$  is a set of spectral 
synthesis.

\end{corollary}
\begin{proof} (i) By Theorem \ref{23}, $$ \cl G(M_{max}(E^*))=M_{max}((\theta \times \theta )^{-1}(E^*))
=M_{max}(\theta ^{-1}(E)^*).$$ Similarly, $$ \cl G(M_{min}(E^*))=M_{min}((\theta \times \theta) ^{-1}(E^*))
=M_{min}(\theta ^{-1}(E)^*).$$ 

(ii) If $E$ is a set of local synthesis, then $M_{max}(E^*)$ is a masa bimodule of operator synthesis. By Corollary \ref{24} 
and (i), $M_{max}(\theta ^{-1}(E)^*) $ is also a masa bimodule of operator synthesis. Thus, by \cite{lt},
 $\theta ^{-1}(E)$  is a set of local synthesis.

(iii) If $A(G)$ possess an approximate identity, then $\theta ^{-1}(E)$  is a set of local synthesis 
if and only if $\theta ^{-1}(E)$  is a set of spectral synthesis. Now use (ii).
\end{proof}

\begin{theorem}\label{25} Let $\cl U\subseteq B(L^2(H))$ be a masa bimodule. If $\cl G(\cl U) $ 
contains a non-zero compact operator, 
then so does $\cl U. $ The same holds replacing compact by finite rank or by rank one operator. 
\end{theorem}
\begin{proof} We have $\cl G(\cl U) =\cl F(R\cl UR)$  and $$R\cl UR=\cl F^{-1}(\cl G(U))=[\cl M^*\cl G(\cl U)\cl M
]^{-w^*}.$$  If $K\in \cl G(\cl U) $ is a non-zero compact operator, then 
$$ \cl M^*K\cl M \subseteq R\cl UR\subseteq \cl U.$$ It suffices to prove that $\cl M^*K\cl M\neq 0.$ 

Suppose that $ \cl M^*K\cl M=0. $ Then $ [\cl M\cl M^*] ^{-w^*} K[\cl M\cl M^*] ^{-w^*}  =0. $ Since 
 $[\cl M\cl M^*] ^{-w^*} =\cl A$ is an unital algebra, $K=0.$  This contradiction shows that $\cl U$
 contains a non-zero compact operator. The remaining cases are proved similarly.

\end{proof}

\begin{corollary}\label{26}  Let $k\subseteq H\times H$ be an $\omega $-closed set
and assume that $ M_{max}((\theta \times \theta )^{-1}(k) )$  (resp. $M_{min}((\theta \times \theta )^{-1}(k))$ 
contains a non-zero compact operator,   
then $M_{max}(k), $ (resp. $M_{min}(k),$) also contains a non-zero compact operator. 
The same holds replacing compact by finite rank or by rank one operator.  
\end{corollary}

\begin{remark}\em{ The implication that if $\cl G( M_{max}(k) )$ contains a non-zero compact operator, 
 then $M_{max}(k) $ also contains a non-zero compact operator, was first proved in \cite[Corollary 4.8]{stwt} 
for some special cases of $\theta .$}
\end{remark}

\begin{theorem}\label{27} Let $I$ be a closed ideal of $A(H)$ and $\cl U=Bim(I^\bot ).$ Then there exists a 
closed ideal $J$ of $A(G)$ such that $\cl G(\cl U)=Bim(J^\bot ).$ 
\end{theorem}
\begin{proof} Let $\rho ^G: G\rightarrow B(L^2(G)), \;\;t\rightarrow \rho ^G_t,$  be 
the right regular representation of $G$ on $L^2(G),$ that is, the representation 
$$ \rho ^G_t (f)(x)=\Delta ^G(x)^{\frac{1}{2}}f(xt),\;\; t,x\;\in \;G, \;f\in L^2(G),  $$
where $\Delta ^G$ is the modular function of $G.$
Similarly, we define the right regular representation $\rho ^H: H\rightarrow B(L^2(H))$ of the group $H.$  
By theorem 4.3 of \cite{akt}  it suffices to prove $$\rho ^G_t  \cl G(\cl U) \rho ^G_{t ^{-1}}\subseteq  \cl G(\cl U), \;\;\forall \;t\;\in \;G
.$$ 
If $P\in L^\infty (H, \nu )$ is a projection, there exists a Borel set $\alpha$ such that $P=P(\alpha )\equiv L^2(\alpha ,\nu ).$  
If $s\in H,$ 
we denote by $P_s$ the projection onto $L^2(\alpha s).$ We can easily see that $\rho _s^HP\rho _{s^{-1}}^H=P_{s^{-1}}.$
 Let $\alpha \subseteq H$ be a Borel set and $t\in G.$ Then 
\begin{align*}& \hat \theta (P_{\theta (t)})= \hat \theta (P(\alpha \theta (t)))=Q(\theta ^{-1}(\alpha \theta (t))) =\\&
Q(\theta ^{-1}(\alpha )t)=Q(\theta ^{-1}(\alpha ))_t=\hat \theta (P(\alpha ))_t=\hat \theta (P)_t,
\end{align*} where $Q(\beta )\equiv L^2(\beta ,\mu  ).$  

Thus if $X\in \cl N, P\in L^\infty (H)$ and $t\in G,$
$$XP_{\theta (t)}=\hat \theta (P)_tX.$$ Therefore,
$$X \rho ^H_{\theta (t)} P\rho ^H_{\theta (t)^{-1}} =XP_{\theta (t)^{-1}}=\hat \theta (P)_{t^{-1}}X=\rho
 ^G_{t} \hat \theta (P)\rho ^G_{\theta (t)^{-1}} X.$$
Also, $$\rho ^G_{t^{-1}}X\rho ^H_{\theta (t)}P=\hat \theta (P)\rho ^G_{t^{-1}}X\rho _{\theta (t)}^H,$$ 
for all $t\in G$ and $ P\in L^\infty (H).$ We conclude that 
$$\rho ^G_{t^{-1}}\cl N \rho ^H_{\theta (t)}\subseteq \cl N.$$
Now take $X, Y\in \cl N, t\in G $ and $ Z\in \cl U.$ There exist $X_1, Y_1\in \cl N$ such that 
$$ \rho _t^GX=X_1\rho _{\theta (t)}^H \;\text{and}\; \rho _t^GY=Y_1\rho _{\theta (t)}^H .$$
Therefore 
$$ \rho _t^GXZY^*\rho _{t^{-1}}^G =X_1\rho _{\theta (t)}^HZ\rho _{\theta (t)^{-1}}^HY_1^*.$$
By theorem 4.3 of \cite{akt}, $\rho ^H_{\theta (t)}Z\rho ^H_{\theta (t)^{-1}} \in \cl U.$ 
Thus $\rho _t^GXZY^*\rho _{t^{-1}}^G \in \cl N\cl U\cl N^*. $ We have proven 
$$\rho ^G_{t}\cl N \cl U\cl N^*\rho ^G_{t^{-1}}\subseteq \cl N \cl U\cl N^*,$$ which implies 
$$\rho ^G_{t}\cl G( \cl U)\rho ^G_{t^{-1}}\subseteq \cl G( \cl U).$$ 
\end{proof}

\begin{remark}\label{29}\em{ If $u\in A(H),$ we denote by $\rho (u)$ the function $u\circ \theta .$ There exist cases of $G, H, \theta $ 
such that $\rho (A(H))\cap A(G)=\{0\}.$ For example if $G$ is a non-compact group $\theta : G\rightarrow H$ is the trivial homomorphism and 
$u(e_H)\neq 0$  then $\rho (u)$ is a non-zero constant map and therefore doesn't belong to $A(G).$ Therefore  
in case $\rho (A(H))\cap A(G)=\{0\}$ if $I$ is a closed ideal of $A(H),$ then $\rho (I)$ is not contained in $A(G).$ Nevertheless, by Theorem \ref{27}, if $\cl U=Bim(I^\bot ),$ 
there is a closed ideal $J\subseteq A(G)$ such that 
$\cl G(\cl U)=Bim(J^\bot ). $ We are going to prove that $Sat(J)= [N(\rho (I))T(G)]^{-\|\cdot \|_t}.$ }
\end{remark}

In the sequel we fix a closed ideal $I\subseteq A(H),$ and write $\cl U=Bim(I^\bot )$ and $\Xi =
[N(\rho (I))T(G)]^{-\|\cdot \|_t}.$ Let $J\subseteq A(G)$  be a closed ideal such that $\cl G(\cl U)=Bim(J^\bot ). $ 
\begin{lemma}\label{210} The space  $\Xi ^\bot $ is a $\cl A$-bimodule.
\end{lemma}
\begin{proof} Let $V_1, V_2\in \cl N, X\in \Xi ^\bot , u\in I$ and $f, g\in L^2(G).$ If 
$N(u)=\sum_i \phi _i\otimes \psi _i,$ we have 
\begin{align*}& (V_1V_2^*XV_3V_4^*, N(u\circ \theta )(f\otimes g))_t= \sum_i(V_1V_2^*XV_3V_4^*, ((\phi _i\circ \theta )f)
\otimes ((\psi _i\circ \theta )g ))_t=
\\& \sum_i(V_2^*XV_3, V_4^*((\phi _i\circ \theta )(f))\otimes V_1^*((\psi _i\circ \theta )(g) ))_t=
\sum_i(V_2^*XV_3, V_4^*(M_{\phi _i\circ \theta}( f))\otimes V_1^*(M_{\psi _i\circ \theta}( g )))_t=\\& 
\sum_i(V_2^*XV_3, M_{\phi _i}V_4^*(f)\otimes M_{\psi _i}V_1^*(( g ))_t= 
\sum_i(X, V_3M_{\phi _i}V_4^*(f)\otimes V_2M_{\psi _i}V_1^*(g ))_t= \\ &
\sum_i(X, M_{\phi _i\circ \theta }V_3V_4^*(f)\otimes M_{\psi _i\circ \theta }V_2V_1^*(g ))_t= 
(X, N(u\circ \theta )(V_3V_4^*(f)\otimes V_2V_1^*(g)))_t
\end{align*}
Since $N(u\circ \theta )(V_3V_4^*(f)\otimes V_2V_1^*(g))\in \Xi $ and $X\in \Xi ^\bot $, we have 
$( V_1V_2^*XV_3V_4^* , N(u\circ \theta )(f\otimes g))_t= 0$ Thus $ V_1V_2^*XV_3V_4^* \in \Xi ^\bot .$ 
The algebra $\cl A$ is equal to $[\cl N\cl N^*]^{-w^*},$ therefore $\cl A\Xi ^\bot \cl A\subseteq \Xi ^\bot .$

\end{proof}

\begin{theorem}\label{211} The spaces  $\Xi $ and $Sat(J)$ are equal.
\end{theorem}
\begin{proof}  First we are going to prove that $$\cl NBim(I^\bot )\cl N^*\subseteq \Xi ^\bot .$$
  Let $V_1, V_2\in \cl N, X\in \Xi ^\bot , u\in I$ and $f, g\in L^2(G).$ If 
$N(u)=\sum_i \phi _i\otimes \psi _i,$ we have 
\begin{align*} & (V_2XV_1^*, N(u\circ \theta )(f\otimes g))_t =\sum_i(V_2XV_1^*, ((\phi _i\circ \theta )f)\otimes ((\psi _i\circ \theta 
)g))_t=\\ & \sum_i(X, V_1^*M_{\phi _i\circ \theta }(f)\otimes V_2^*M_{\psi _i\circ \theta }(g))_t=
\sum_i(X, M_{\phi _i}V_1^*(f)\otimes M_{\psi _i}V_2^*(g))_t=\\ & (X, N(u)(V_1^*(f)\otimes V_2^*(g)))_t.
\end{align*}
Since $X\in Bim(I^\bot )$ and $u\in I$, we have 
$$(V_2XV_1^*, N(u\circ \theta )(f\otimes g))_t =0.$$
Thus $$\cl NBim(I^\bot )\cl N^*\subseteq \Xi ^\bot \Rightarrow Bim(J^\bot )\subseteq \Xi ^\bot \Rightarrow \Xi \subseteq Sat(J).$$
If $X\in \Xi ^\bot $ and $V_1, V_2, V_3, V_4\in \cl N,$ then Lemma \ref{210} implies that
$$V_1V_2^*XV_3V_4^*\in \Xi ^\bot .$$ Thus for all $u\in I,f, g \in L^2(G),$ we have 
$$0=(V_1V_2^*XV_3V_4^*, N(u\circ \theta )(f\otimes g))_t=(V_2^*XV_3, N(u)(V_1^*(f)\otimes V^*_4(g)))_t.$$
Since $$R(L^2(H))=\overline{[\cl M^*(L^2(G))]},$$ we conclude that 
$$0=(V_2^*XV_3, N(u|_{\alpha_0\times \alpha_0}(f\otimes g)), \;\forall \;f,g \;\in L^2(\alpha _0),\;u\in I.$$
Since $$RBim(I^\bot )R=[N(u|_{\alpha_0\times \alpha _0}(f\otimes g): u\in I, \;\;f,g \;\in L^2(\alpha _0)]^\bot ,$$
we have that $V_2^*XV_3\in  RBim(I^\bot )R .$ Therefore 
$$\cl N^*\Xi ^\bot \cl N\subseteq RBim(I^\bot )R ,$$ which implies 
$$\cl N\cl N^*\Xi ^\bot \cl N\cl N^*\subseteq \cl F(RBim(I^\bot )R )=\cl G(Bim(I^\bot ))=Bim(J^\bot ).$$ 
The space $\cl A=[\cl N\cl N^*]^{-w^*}$ is an unital algebra, thus 
$$\Xi ^\bot \subseteq Bim(J^\bot )\Rightarrow Bim(J^\bot )^\bot \subseteq \Xi \Rightarrow Sat(J)\subseteq \Xi .$$
Since we have already shown  $\Xi \subseteq Sat(J)$, we obtain the required equality.
\end{proof}

For the following theorem, we recall from \cite{tt} that a closed ideal $J\subseteq A(G)$ 
is an ideal of multiplicity if $J^\bot \cap C_r^*(G)\neq \{0\},$ where $C_r^*(G)$ is the reduced $C^*$-algebra of $G.$

\begin{theorem}\label{212} Let $I$ be a closed ideal of $A(H).$ By Theorems \ref{27} and \ref{211}  there exists 
 a closed ideal $J\subseteq A(G)$ such that 
$$\cl G(Bim(I^\bot ))= Bim(J^\bot ) \;\text{and} \;Sat(J)=[N(I)(T(G))]^{-\|\cdot \|_t}.$$ 
If $J$ is an ideal of multiplicity, then $I$ is also an ideal of multiplicity.

\end{theorem}
\begin{proof}  By \cite [Corollary 1.5 ]{tt}, if $J$ is an ideal of multiplicity, then $Bim(J^\bot )$ contains a non-zero compact operator. 
By Theorem \ref{25},
 $Bim(I^\bot )$ contains a non-zero compact operator. Thus, again by \cite[Corollary 1.5 ]{tt}, $I$ is an ideal of multiplicity.
\end{proof}

A closed set $E\subseteq H$ is called an $M$-set (resp. an $M_1$-set) if the ideal $J_H(E)$ (resp. $I_H(E)$) is an ideal 
of multiplicity. Corollaries \ref{28} (i) and \ref{26} together with \cite[Corollary 3.6]{tt} imply the following:

\begin{corollary} If $E\subseteq H$ is a closed set such that $\theta ^{-1}(E)$ is an $M$-set (resp. an $M_1$-set), 
then  $E$ is an $M$-set (resp. an $M_1$-set). 
\end{corollary}

\begin{remark}\em{ The previous corollary was proven in \cite{stwt} for some special cases of $\theta .$}
\end{remark}

\section{The case when $ \theta _*(\mu ) $ is a Haar measure for $\theta (G)$}\label{3333}

Let $G $ and $ H$ be locally compact, second countable groups with Haar measures $\mu$ and $ \nu $ respectively. Suppose that $\theta : G\rightarrow H$ 
is a continuous homomorphism, and assume that $m=\theta _*(\mu ) <<\nu .$ Since $G$ is a $\sigma -$compact set and $\theta $ 
is a continuous map then $\theta (G)$ is also a $\sigma -$compact set and hence a Borel set. Also  $\theta _*(\mu ) <<\nu $ 
implies that $\nu (\theta (G))>0.$ By Steinhaus theorem the group $\theta(G)$ contains an open set. We conclude that $\theta(G)$ 
is an open set. We note that the open subgroups of a locally compact group are closed. Using these facts we can easily see that $\nu |_{H_0}$ is a Haar measure of $H_0=\theta (G).$ 
In some cases $m <<\nu $ implies that $m$ is a Haar measure for $H_0.$ Thus there exists $c>0$ such that $m|_{H_0}=c \nu |_{H_0}.$ In 
this section we investigate this equality. We can replace the Haar measure $\nu $ with $c\nu $ and thus we may assume that 
$m(\alpha )=\nu (\alpha )$ for all Borel sets $\alpha \subseteq H_0.$ 

For every $u\in A(H),$ we define $\rho (u)=u\circ \theta .$ We are going to prove that $\rho : A(H)\rightarrow A(G)$ 
is a continuous homomorphism and that if $I$ is a closed ideal of $A(H)$ and $ \cl U=Bim(I^\bot )$, then 
$\cl G(\cl U)=Bim(\rho _*(I)^\bot ),  $ where $\rho _*(I)$ is the closed ideal of $A(G)$ generated by $\rho (I)$ 
and $\cl G$ is the map defined in Section \ref{2222}. For every $u\in A(H),$ we denote by $\pi (u)$ 
the function $u|_{H_0}.$ By \cite[Theorem 2.6.4] {kan}, $\pi (u)\in A(H_0)$ for all $u\in A(H).$ Thus if $u\in A(H),$ 
there exist $f, g\in L^2(H_0)$ such that 
$$u(t)=\pi (u)(t)=(\lambda _t^{H_0}f, g),\;\;\forall \;t\;\in \;H_0.$$ By \cite [Corollary 2.6.5]{kan}, 
the map $\pi : A(H)\rightarrow A(H_0)$ is contractive onto homomorphism, thus $\pi (I)$ is an ideal of $A(H_0).$
Also, observe that the map $A: L^2(H_0)\rightarrow L^2(G)$ given by $A(f)=f\circ \theta $ is an isometry.

\begin{lemma}\label{311} Let $u\in A(H).$ Then $\rho (u)\in A(G).$ Actually, if 
$$u(t)=\pi (u)(t)=(\lambda _t^{H_0}f, g),\;\;\forall \;t\;\in \;H_0,$$ 
then $$\rho (u)(s)= (\lambda _s^GAf, Ag) ,\;\;\forall \;s\;\in \;G.$$
\end{lemma}
\begin{proof} For every $s\in G,$ we have 
\begin{align*}&(\lambda _s^GAf, Ag)= \int_G Af(s^{-1}t) Ag(t)d\mu (t) =
\int_G (f\circ \theta )(s^{-1}t) (g\circ \theta )(t)d\mu (t) =\\ &
\int_G (f_ {\theta (s^{-1})} \circ \theta )(t) (g\circ \theta )(t)d\mu (t) = \int _{H_0}f_{\theta (s^{-1})} 
(t)g(t)dm(t) =\\ &\int _{H_0}f_{\theta (s^{-1})} (t)g(t)d\nu (t)=(\lambda _{\theta (s)}^{H_0}f, g) =u(\theta (s)).
\end{align*}
\end{proof}
\begin{theorem}\label{312} The map $\rho : A(H)\rightarrow A(G)$ is a continuous homomorphism.
\end{theorem}
\begin{proof} If $u_1, u_2\in A(H),$ 
then
$$ \rho (u_1u_2) =(u_1u_2)\circ \theta =u_1\circ \theta \cdot u_2\circ \theta =\rho (u_1) \rho (u_2). $$
Let $u\in A(H).$ We assume that $$f, g \in L^2(H_0),\;\;\pi (u)(t)=(\lambda _t^{H_0}f, g),\;\;\forall \;t\;\in \;H_0.$$
By Lemma \ref{311}, $$\rho (u)(s)= (\lambda _s^GAf, Ag) ,\;\;\forall \;s\;\in \;G.$$
Thus $$\|\rho (u)\|_{A(G)}\leq  \|Af\|_2\|Ag\|_2 =\|f\|_2\|g\|_2 $$ for all $f$ and $g$ such that 
$\pi (u)(t)=(\lambda _t^{H_0}f, g).$ Thus $$\|\rho (u)\|_{A(G)}\leq \|\pi (u)\|_{A(H_0)}.$$ 
Since $\pi $ is a contraction, $$\|\rho (u)\|_{A(G)}\leq \|u\|_{A(H)}.$$ 
\end{proof}

\begin{lemma}\label{313} Let $I\subseteq A(H)$ be an ideal and $X\in B(L^2(G))$ such that 
$$(X, N(\rho (u))h)_t=0,\;\;\;\forall \;h\;\in \;T(G),\;u\;\in \;I.$$
Then $$ (X, N(v)h)_t=0 ,\;\;\;\forall \;h\;\in \;T(G),\;v\;\in \;\rho _*(I).$$
\end{lemma}
\begin{proof} Let 
$$K=\{v\in \rho _*(I): (X, N(v)h)_t=0 ,\;\;\;\forall \;h\;\in \;T(G)\}.$$ Clearly $K$ is a closed subset of $A(G)$ and  
 $\rho(I) \subseteq K\subseteq \rho_*(I) .$ If $v_1\in K$ and$ v_2\in A(G),$ we have 
$$(X, N(v_1v_2)h)_t=(X, N(v_1)( N(v_2)(h))_t.$$ Since $N(v_2)(h)\in T(G), \;\;v_1\in K$, we have 
$(X, N(v_1v_2)h)_t=0.$ Thus $v_1v_2\in K.$ Therefore $K$ is an ideal. We conclude that $K=\rho _*(I).$
\end{proof} 
In the following theorem, we fix a closed ideal $I\subseteq A(H),$ we assume that $\cl U=Bim(I^\bot )$ 
and define $ \cl G(\cl U) =[\cl N\cl U\cl N^*]^{-w^*},$ where $\cl N$ is the TRO defined in Section \ref{2222}. 
We are going to prove that $\cl G(\cl U) =Bim(\rho _*(I)^\bot ).$

\begin{theorem}\label{314} The space $\cl G(\cl U) $ is equal to $Bim(\rho _*(I)^\bot ).$
\end{theorem}
\begin{proof} By Theorem \ref{211} there exists a closed ideal $J$ of $A(G)$ such that 
$\cl G(\cl U) =Bim(J^\bot )$ and $ Sat(J) =[N(u\circ \theta )h: u\in I, h\in T(G)]^{-\|\cdot\|_t}.$
Clearly $$Sat(J) \subseteq Sat(\rho _*(I))\Rightarrow Bim(\rho _*(I)^\bot )\subseteq Bim(J^\bot )= \cl G(\cl U) .$$
We need to prove $$\cl G(\cl U) \subseteq Bim(\rho _*(I)^\bot ).$$
It suffices to prove $$\cl N \cl U \cl N^* \subseteq Bim(\rho _*(I)^\bot ).$$
Let $X\in \cl U,\;\;V_1,V_2\in \cl N,\;u\in I.$ Assume that $N(u)=\sum_i\phi _i\otimes \psi _i.$ 
For all $f, g \in L^2(G)$ we have 
\begin{align*} & (V_1XV_2^*, N(u\circ \theta )(f\otimes g))_t =\sum_i(V_1XV_2^*, (\phi _i\circ \theta )
f\otimes (\psi _i\circ \theta )g)_t =\\
&\sum_i(X, V_2^*( (\phi _i\circ \theta )
f )\otimes V_1^*(\psi _i\circ \theta )g))_t .\end{align*}
We have $$V_2^*((\phi _i\circ \theta )f )=V_2^*M_{\phi _i\circ \theta }(f)= M_{\phi _i}V_2^*(f) .$$
Similarly, $$V_1^*((\psi _i\circ \theta )g)= M_{\psi _i}V_1^*(g)  .$$
Therefore 
$$(V_1XV_2^*, N(u\circ \theta )(f\otimes g))_t =\sum_i(X, M_{\phi _i} V_2^*(f) \otimes M_{\psi _i} V_1^*(g)   )_t=
(X, N(u)(V_2^*(f) \otimes V_1^*(g)   ))_t=0,$$
because $X\in Bim(I^\bot )$ and $u\in I.$
Therefore 
$$ (V_1XV_2^*, N(u\circ \theta )(h))_t =0\Rightarrow 
(V_1XV_2^*, N(\rho (u))(h))_t =0,\;\;\forall \;h\;\in \;T(G), \;u\;\in \;I$$
By Lemma \ref{313}, we have $$
(V_1XV_2^*, N(v)(h))_t =0,\;\;\forall \;h\;\in \;T(G), \;v\;\in \;\rho _*(I).$$ Thus 
$V_1XV_2^*\in Bim(\rho _*(I)^\bot ),$ which implies  $$\cl N \cl U\cl N^* \subseteq Bim(\rho _*(I)^\bot ).$$
\end{proof} 
\begin{theorem}\label{315} Let $I$ be a closed ideal of $A(H).$ If $\rho _*(I)$ is an ideal of multiplicity then $I$ is also an ideal of multiplicity.

\end{theorem}

The proof of the above theorem is consequence of Theorems \ref{212} and \ref{314}.

In the last part of this section we will prove that if $\theta _*(\mu )$ is a Haar measure 
for $\theta (G)$  and $A(G)$ contains a (possibly unbounded) 
identity and $E\subseteq H$ is an ultra strong Ditkin set, then $\theta ^{-1}(E)$ is also an ultra strong Ditkin set.

\begin{definition}\label{gias1} Let $E\subseteq H$ be a closed set. We call $E$ an ultra strong Ditkin set if there exists a bounded 
net $(u_\lambda)\subseteq J_H^0(E)$ such that 
$u_\lambda u\rightarrow u,$ for every $u\in I_H(E).$ 
\end{definition}

\begin{lemma}\label{gias2} Let $E\subseteq H$ be a closed set. Then $$\rho (J_H^0(E))\subseteq J_G^0(\theta ^{-1}(E)).$$
\end{lemma}
\begin{proof}If $u\in J_H^0(E)$, there is an open set $\Omega \subseteq H$ such that $E\subseteq \Omega $ and $u|_{\Omega }=0.$ 
We consider the open set $ \theta ^{-1}(\Omega ) .$ Since $u\circ \theta |_{\theta ^{-1}(\Omega ) }=0$ and $ \theta ^{-1}(E )  \subseteq 
\theta ^{-1}(\Omega ), $ we conclude that $\rho (u)=u\circ \theta \in J_G^0(\theta ^{-1}(E)).$
\end{proof}

\begin{lemma}\label{gias4} Let $E\subseteq H$ be a closed set and suppose $(u_\lambda)\subseteq J_H^0(E )$ is a bounded net such that 
$u_\lambda u\rightarrow u$ for every $u\in I_H(E).$ Then $\rho (u_\lambda)v\rightarrow v$ for every $v\in \rho _*(I_H(E)).$ 
\end{lemma}
\begin{proof} Define the space 
$$I=\{ v\in \rho _*(I_H(E)) : v=\lim_\lambda\rho (u_\lambda)v \}. $$ If $u\in I_H(E)$ we have  $\rho (u_\lambda)\rho (u)\rightarrow \rho (u).$  Then $\rho (I_H(E))\subseteq I.$ 
 If $v_1\in I, v_2\in A(G),$ we have 
$$\rho (u_\lambda)v_1\rightarrow v_1\Rightarrow \rho (u_\lambda)v_1v_2\rightarrow v_1v_2,$$ 
thus $v_1v_2\in I.$ Therefore  $I$ is an ideal. Since $(\rho (u_\lambda))$ is a bounded net we can easily see that 
if $(v_i)\subseteq I$ is a sequence such that $v_i\rightarrow v$, then $\lim_\lambda \rho (u_\lambda)v=v.$ Thus $I$ is a closed ideal, which implies $I=\rho _*(I_H(E)) .$ 
The proof is complete. 
 
\end{proof}

\begin{theorem}\label{gias5} Let $E\subseteq H$ be a closed set and assume that $A(G)$ has a (possibly unbounded) approximate 
identity.  If $E$ is an ultra strong Ditkin set, then  $\theta ^{-1}(E)$ is an ultra strong Ditkin set.

\end{theorem}
\begin{proof}

 Theorem 5.3 of \cite{akt} implies that 
$$   M_{min}(\theta ^{-1}(E)^*)  = Bim(I_G(\theta ^{-1}(E))^\bot ).$$ If $E$ is an ultra strong  Ditkin set, then by Corollary \ref{24} the set $\theta ^{-1}(E)^*$ 
is operator synthetic. Thus $$M_{max}(\theta ^{-1}(E)^*)=  M_{min}(\theta ^{-1}(E)^*)=  Bim(I_G(\theta ^{-1}(E))^\bot ).$$
From Theorem \ref{314}, we have $$M_{max}(\theta ^{-1}(E)^*)= Bim(\rho _*(I_H(E)^\bot ).$$
Lemma 4.5 in  \cite{akt} implies that 
\begin{equation}\label{kke}I_G(\theta ^{-1}(E))=\rho _*(I_H(E)).\end{equation}  
Now Lemmas \ref{gias2} and \ref{gias4} together with (\ref{kke}) imply that $\theta ^{-1}(E)$ is also an ultra strong Ditkin set.

\end{proof}

\bigskip

\noindent {\bf Acknowledgement. } 
We thank the anonymous referee for useful suggestions that led to the improvement of the presentation.

\end{document}